\documentclass{amsart}
\usepackage{amssymb}
\usepackage{amsmath}
\usepackage{kotex}
\usepackage{colortbl}

\newtheorem{theorem}{Theorem}[section]
\newtheorem{lemma}[theorem]{Lemma}

\theoremstyle{definition}
\newtheorem{definition}[theorem]{Definition}

\theoremstyle{remark}
\newtheorem{remark}[theorem]{Remark}

\numberwithin{equation}{section}

\newcommand{ \mint }{ {\int\hspace{-0.38cm}- }}

\newcommand{ \R }{ \mathbb{R} }

\newcommand{ \ba }{ \mathbf{a}}
\newcommand{ \loc }{ \mathrm{loc}}

\begin{document}

\title[]
{$L^q$-regularity for nonlinear elliptic equations with Schr\"odinger-type lower order terms}


\author{Mikyoung Lee}
  \address{Mikyoung Lee, Department of Mathematics, Pusan National University, Busan
46241, Republic of Korea}
\email{mikyounglee@pusan.ac.kr}

\author{Jihoon Ok}
\address{Jihoon Ok, Department of Mathematics, Sogang University, Seoul 04107, Republic of Korea}
\email{\texttt{jihoonok@sogang.ac.kr}}

\thanks{M. Lee was supported by the National Research Foundation of Korea (NRF) grant funded by the Korea Government (NRF-2019R1F1A1061295). J. Ok was supported by the National Research Foundation of Korea funded by the Korean
Government (NRF-2017R1C1B2010328).}





\begin{abstract} We consider nonlinear elliptic equations of the $p$-Laplacian type  with lower order terms which involve nonnegative potentials satisfying a reverse H\"older type condition. Then we obtain interior and boundary $L^q$ estimates for the gradient of weak solutions and the lower order terms, independently, under sharp regularity conditions on the coefficients and the boundaries. 
In particular, the proof in this paper does not employ Fefferman-Phong type inequalities which are essential tools in the linear cases in \cite{AB1,Sh1}.
\end{abstract}

\maketitle



\section{Introduction}\label{sec1}

In this paper we study $L^q$-regularity theory for the following nonlinear equations of the $p$-Laplacian type with lower order terms:
 \begin{equation}
\label{maineq}
\left\{\begin{array}{rclcc}
-\mathrm{div}\, \mathbf{a}(x,Du) + V |u|^{s-2}u&  =  & -\mathrm{div}\, (|F|^{p-2}F) & \textrm{ in } & \Omega,  \\
u & = & 0 & \textrm{ on } & \partial \Omega,
\end{array}\right.
\end{equation}
where $1<p<\infty$, $1 < s  < p^*$ (see \eqref{gamma*} with $\gamma=p$),
$\Omega$ is a bounded and open set in $\R^n$ with $n\geq2$,
$V:\R^n\to \R$ is nonnegative and called a \textit{potential}, and $F\in L^p(\Omega,\R^n)$. The nonlinearity $\mathbf{a} : \R^n \times \R^n \rightarrow \R^n $ is assumed to be a vector-valued Carath\'eodory function (i.e., $ \mathbf{a}$ is measurable in the $x$-variable and continuous in the $\xi$-variable) of $p$-Laplacian type whose prototype is
\begin{equation}\label{model}
\mathbf{a}(x,\xi) = (A(x)\xi\cdot\xi)^{\frac{p-2}{2}}A(x)\xi,
\end{equation}
where $A:\R^n\to \R^{n^2}$ is an $n\times n$ matrix satisfying that  
$$
\nu |\eta|^2 \leq A(x)\eta\cdot\eta 
\  \text{ and }\  |A(x)| \le L, \quad x,\eta\in\R^n,
$$
for some $0<\nu\le L$.  

For equations \eqref{maineq} with identically zero potential function, i.e. $V\equiv0$, the $L^q$-regularity theory has been extensively studied after pioneering work of Calder\'on and Zygmund in \cite{CZ1} where $L^q$-regularity estimates for  Poisson equations $-\Delta u=f$ or $\mathrm{div} F$ were proved. 
In particular, Byun and Wang established the global $L^q$-regularity for the linear equations with Bounded Mean Oscillation(BMO) coefficients in Reifenberg flat domains \cite{BW1}. 
We also refer to \cite{Mey1,D2} and references therein for more $L^q$-regularity results of the linear equations.  
With regard to nonlinear equations of the $p$-Laplacian type, Iwaniec \cite{Iw1} first obtained $L^q$-regularity estimates 
for the simplest case $\ba(x,\xi)=|\xi|^{p-2}\xi$,
 and Caffarelli and Peral \cite{CP1} considered general $\ba(x,\xi)$ that can be discontinuous for $x$-variable, see also \cite{BPS1,DM1,LO1,KZ1} for further results. 
Based on those works, $L^q$-regularity theory  has been actively developed for various equations generalized from the $p$-Laplace equations, for instance, parabolic equations of the $p$-Laplacian \cite{AM1,BOR0,Bo1}, elliptic equations with nonstandard growth \cite{AM0,BOh1,BOR1,CM1}.

On the other hand, in the case $V\not\equiv 0$, Shen \cite{Sh1} obtained various $L^q$ estimates for 
the following linear elliptic equation:
\begin{equation}\label{Schrodingereq}
(-\Delta+V)u = -\Delta u + Vu= \mathrm{div} F \ \ (\text{or}\ \ f) 
\end{equation}
i.e. with the right hand side in divergence form or in nondivergence form.
Here, the linear operator $-\Delta +V$ is called the Schr\"odinger operator since the above equations were motivated by the (normalized) Schr\"odinger equation
\[
i u_t = -\Delta u +Vu.
\]
In particular, in \cite{Sh1}, it is shown that if the potential function $V$ satisfies the reverse H\"older type condition in \eqref{VBqclass} 
(i.e., $V\in \mathcal B_\gamma$)
 for some $\gamma\ge \frac{n}{2}$, then the following estimates hold:
\begin{equation}\label{shen1}
\begin{cases}
\|Du\|_{L^{2q}(\R^n)} \le c \|F\|_{L^{2q}(\R^n)}, \quad &\frac{(\gamma^*)'}{2}\le q\le \tfrac{\gamma^*}{2} \ (\tfrac 12<q<\infty \text{ if }\gamma\ge n),\\
\|V^{\frac{1}{2}}|u|\|_{L^{2q}(\R^n)} \le c \|F\|_{L^{2q}(\R^n)}, \quad& \frac{(\gamma^*)'}{2}\le q\le \gamma \ (\frac{1}{2}<q<\gamma \text{ if }\gamma\ge n),
\end{cases}
\end{equation}
where the constants $c>0$ depend only on $n$, $q$, $\gamma$, and $b_\gamma$. Later, Auscher and Ben Ali \cite{AB1} extended the range of $\gamma$  such that $\gamma>1$ and proved  the first estimate in \eqref{shen1} whenever $\frac{1}{2}< q\le \max\{\frac{\gamma^*}{2}, \gamma\}$ and the second 
one
whenever $\frac{1}{2}< q\le  \gamma$, by improving the techniques used in \cite{Sh1} and applying the regularity results for the equation \eqref{Schrodingereq}. 
In addition, Shen's results in \cite{Sh1} have been extended to linear equations with variable coefficients by  Bramanti,  Brandolini,  Harboure and Viviani \cite{BBHV1} and Pan and Tang \cite{PT1}.
We further refer to  \cite{CFG1,D1,Ka1,Ka2,Se1,Sh0,K1,Zho1} and references therein
for regularity theory relevant to the Schr\"odinger type elliptic equations.  

From the equation \eqref{Schrodingereq}, it is natural to consider the following semi-linear equation: 
\begin{equation}\label{semilineareq}
 -\Delta u + V|u|^{s-2}u= \mathrm{div} F \ \ (\text{or}\ \ f),
\end{equation}
where $1<s<2^*$. For the basic theory of semi-linear elliptic equations, we refer to, for instance, \cite{BS1} and references therein.  The equation \eqref{semilineareq} is the Euler-Lagrange equation of the following energy functional:
\[
v\in W^{1,2}(\Omega)\ \ \mapsto\ \ \int_\Omega |Dv|^{2} + V|v|^s+F\cdot Dv\, dx.
\]
We also note that  time-independent inhomogeneous nonlinear Schr\"odinger(INLS) equations reduce to the equation \eqref{semilineareq} with $F\equiv 0$. We refer to  \cite{Din1,FW1,KLS1,Mer1,RS1} and references therein for INLS equations and \cite{Be1,Gil1,LT1,TM1} for their physical applications.  
For the equations \eqref{semilineareq}, however, in our best knowledge, no systematic $L^q$ regularity result has been reported.

Our main equation \eqref{maineq} is a generalized version of \eqref{semilineareq}. Indeed, the equation \eqref{maineq} with \eqref{model} is the Euler-Lagrange equation of 
\[
v\in W^{1,p}(\Omega)\ \ \mapsto\ \  \int_\Omega (A(x)Dv\cdot Dv)^{\frac{p}{2}}+ V|v|^s+ |F|^{p-2}F\cdot Dv \,dx, \quad 1<p<\infty.
\] 
Recently, for the equation \eqref{maineq} in the case $s=p$ with $V\in \mathcal B_\gamma$ and  $\frac{n}{p}<\gamma < n$, the authors \cite{LO2} derived the following local and global $L^q$ estimates for $1\le q \le \frac{\gamma^*(p-1)}{p}$:
\[
\mint_{B_{r}}\left[|Du|^p+V|u|^{p}\chi_{\{q\le \gamma\}} \right]^q\,dx  \le c \left( \mint_{B_{2r}}|Du|^p+V|u|^{p}\,dx\right)^q +c\mint_{\Omega_{2r}}|F|^{pq}\,dx
\]
with $B_{2r}\Subset\Omega$ and $r\le R_0$, and 
\[
\|Du\|_{L^{pq}(\Omega)} + \|V^{\frac{1}{p}}|u|\|_{L^{pq}(\Omega)} \chi_{\{q\le \gamma \}} \le c_0 \|F\|_{L^{pq}(\Omega)}.
\]
Here, $\chi_{\{q\le\gamma\}}:=1$ or $0$ when $q\le \gamma$ or $q> \gamma$, respectively, and  $c_0$ and $R_0$ depend on $\|V\|_{L^\gamma(\Omega)}$. 
We note that the ranges of $\gamma$ and $q$ are extended from the ones in \eqref{shen1} when $p=2$. (In fact, a naturally extended range could be $\frac{(\gamma^*)'(p-1)}{p}<q\le  \frac{\gamma^*(p-1)}{p}$, but the case $q<1$ is a famous open problem even when $V\equiv 0$, see \cite{Iw1}.) However, these results do not cover the ones in \cite{AB1}. In addition, the above resulting estimates are not sharp since $R_0$ and $c_0$ depend on $\|V\|_{L^\gamma}$, whereas the estimates for the linear case in \cite{AB1,Sh1} are independent of  $\|V\|_{L^\gamma}$. Moreover, the estimates were derived for $|Du|^p$ and $V|u|^p$ considered together.
We also refer to \cite{TNN1} for regularity estimates in the Lorentz spaces.

In this paper, we establish $L^q$ estimates for $Du$ and $V|u|^s$ with possibly sharp ranges of the exponents $\gamma$ and $q$. In particular, we deal with the estimates for $Du$ and $V|u|^s$, separately, and find both local and global estimates that are independent of $\|V\|_{L^\gamma}$.
Our results extend the known linear regularity results, especially in \cite{AB1} to nonlinear setting. Furthermore,  
a BMO, possibly discontinuous, nonlinearity $\ba$ for $x$-variable and non-smooth domain $\Omega$ that beyonds the Lipschitz category are considered as our regularity assumptions.

  The main difficulty is that we cannot take advantage of the techniques for linear operators that were used in \cite{Sh1,AB1}, since we deal with nonlinear problems. Instead, we apply various estimates and techniques used in the study of the regularity theory  to our problems. For instance,  for \eqref{maineq} with $F\equiv 0$, we employ $L^\infty$ estimates and Calder\'on-Zygmund estimates, and use an iteration argument. As a consequence, additionally using the reverse H\"older condition of the potential function $V$, we obtain reverse H\"older type inequalities for $|Dh|^p$ and $V|h|^s,$ where $h$ is a weak solution to a localized equation of \eqref{maineq} with $F\equiv 0$ (see Theorems~\ref{lem:estimatehomo0} and \ref{lem:estimatehomo}). In particular, we stress that we do not make use of Fefferman-Phong inequality in \cite{Fe} or its variation that plays an important role in the proofs of $L^q$ estimates in \cite{AB1,LO2,Sh1}  and in turn, the approach in this paper is simpler than eariler ones. Moreover, we obtain comparison estimates for the gradient of solutions and the lower order terms separately (see  Lemma~\ref{lem:approximation}). Therefore, we can handle them independently in the final proof of $L^q$ estimates to discover better resulting estimates.

 The remaining is organized as follows. In the next section, we state our main results. Section 3 contains various regularity estimates for the homogenous equations with auxiliary lemmas. Lastly, we prove our main results in Section 4.



\section{\bf  Main result }
\label{secpre}

\subsection{Preliminaries}
We start with standard notation and definitions. 
We write  $B_r(y)$ for  the open ball in $\R^n$ with center $y\in \R^n$ and radius $r>0$. We denote $\Omega_r (y)= B_r(y) \cap \Omega$ and $  \partial_{\mathrm{w}}\Omega_{r}(y) = B_r(y) \cap \partial \Omega.$
 For the sake of simplicity, we write $B_r=B_r(0)$ 
and $\Omega_{r} = \Omega_{r}(0).$
For a measurable function $g:U\to \R$ with $U\subset\R^n$, we define
$$ (g)_U:=\mint_{U} g \; dx = \frac{1}{|U|} \int_{U} g \;dx, $$
$$
g_+:=\max\{g,0\} \quad \text{and}\quad g_-:= (-g)_+=\max\{-g,0\}.
$$ 
For $g\in W^{1,p}(\Omega_r(y))$, the boundary condition ``$g=0$ on $\partial_{\mathrm{w}}\Omega_r$" means the zero extension of $g$ to $B_r(y)$ is in $W^{1,p}(B_r(y))$. We also define
\begin{equation}\label{gamma*}
\gamma^*:=
\begin{cases}
\frac{n\gamma}{n-\gamma}, & \quad \text{when }\ 1<\gamma <n,\\
\infty, & \quad \text{when }\ \gamma \ge n.
\end{cases}
\end{equation}

We say that a nonnegative function $V:\R^n\to [0,\infty)$ belongs to $\mathcal{B}_\gamma$ for some $ \gamma >1$ if $V\in L^{\gamma}_{\loc}(\R^n)$ and there exists a constant $b_{\gamma}>0$ such that the \textit{reverse H\"older inequality}
\begin{equation}\label{VBqclass}
\left( \frac{1}{|B|} \int_{B} V^{\gamma} \, dx \right)^{\frac{1}{\gamma}} \leq b_\gamma \left( \frac{1}{|B|} \int_{B} V \, dx \right)
\end{equation}
holds for every ball $B$ in $\R^n.$  This  $\mathcal{B}_\gamma$ class, which is a wide class including all nonnegative polynomials, was introduced independently by Muckenhoupt \cite{Mu} and Gehring \cite{Ge} in the study of weighted norm inequalities and quasi-conformal mapping, respectively. One notable example of this element is $ V(x) = |x|^{-n/ \gamma} $
which actually belongs to the $\mathcal{B}_{\tilde{\gamma}}$ class for all $\tilde{\gamma} < \gamma.$ Moreover, the $\mathcal{B}_\gamma$ class is strongly connected to the Muckenhoupt classes
%
%
$A_p$.
We say that a nonnegative function $w\in L^1_{\loc}(\R^n)$ is in  the $A_p$ class for some $1\leq p<\infty$, denoted by $w\in A_p$,  if and only if
$$
[w]_p:=\sup_B \left(\mint_Bw\, dx\right)\left(\mint_B w^{-\frac{1}{p-1}}\, dx\right)^{p-1}<\infty,
$$
where the supremum is taken over all balls $B \subset \R^n.$  
Then, we have the following equivalent relations:
\begin{equation}\label{equivalence}
V\in A_p \textrm{ for some } p>1\ \  \Longleftrightarrow \ \    V\in \mathcal B_{\gamma} \textrm{ for some }\gamma>1
\end{equation}
Here, if we  consider only the left arrow, 
the constant $p$ and $[V]_p$ are determined by  $\gamma$ and $b_\gamma$.
We refer to \cite[Theorem 9.3.3]{G1} for its proof and more details on properties and relations of those classes.

\subsection{Main result}

We introduce the main result in this paper.

We first recall the equation \eqref{maineq} with the basic setting in the first paragraph in Introduction, Section~\ref{sec1}.
The nonnegative potential $V:\R^n\to[0,\infty)$ satisfies that $V\in L^{\gamma_0}(\Omega)$, where 
\begin{equation}\label{gamma0}
\gamma_0 := 
\begin{cases}
\frac{np}{np-s(n-p)},  &\text{when }\ 1<p<n,\\
\text{any number larger than }1,\quad &\text{when }\ p=n,\\
1, &\text{when }\ p>n.
\end{cases}
\end{equation}
Note that 
$\gamma_0=\frac{n}{p}$ if $s=p\in (1,n)$.
%
The
nonlinearity $\ba(x,\xi)$
 is assumed that $\ba(x,\cdot)\in C^1(\R^n\setminus\{0\},\R^n)$ for each $x\in\Omega$ and satisfies the following growth and ellipticity conditions:
\begin{equation}\label{aas1}
| \mathbf{a}(x,\xi)|+ | D_{\xi}  \mathbf{a}(x,\xi)||\xi| \leq L |\xi|^{p-1}
\end{equation}
and
\begin{equation}\label{aas2}
D_{\xi}  \mathbf{a}(x,\xi)\, \eta \cdot  \eta \geq \nu |\eta|^2|\xi|^{p-2}
\end{equation}
for almost all $x \in \R^n$ and any $\xi, \eta \in \R^n$($\xi\neq0$) and for some constants $L, \nu$ with $0< \nu \leq 1 \leq L.$  We remark that the  condition \eqref{aas2} implies the monotonicity condition:
\begin{equation}\label{mono}
\left( \mathbf{a}(x,\xi) - \mathbf{a}(x,\eta) \right) \cdot (\xi-\eta) \geq c(p,\nu) \left( |\xi|^2 + |\eta|^2 \right)^{\frac{p-2}{2}} |\xi-\eta|^2
\end{equation}
for any $\xi, \eta \in \R^n$ and a.e. $x \in \R^n.$
Under the above setting,  we say that  $u \in W^{1,p}_0(\Omega)$ is a  weak solution to the problem \eqref{maineq} if
\begin{equation*}
 \int_{\Omega} \mathbf{a}(x, Du) \cdot D\varphi \, dx  + \int_{\Omega} V |u|^{s-2}u  \varphi  \, dx =\int_{\Omega} |F|^{p-2} F \cdot D \varphi\, dx
\end{equation*}
holds for any $\varphi \in W_0^{1,p}(\Omega).$  
Under the above setting,
the existence and the uniqueness of the weak solution of \eqref{maineq} follow from the theory of nonlinear functional analysis, see for instance \cite[Chapter 2]{Sho1}.


The following two definitions are related to our main assumptions imposed on the nonlinearlity  $\mathbf{a}$ and the domain $\Omega.$
\begin{definition} \label{smallbmo} We say that $\mathbf{a}:\R^n\times\R^n \to\R^n$ is \textit{$(\delta,R)$-vanishing} if
$$
\sup_{0<\rho\leq R}  \ \sup_{y\in\mathbb{R}^n} \mint_{B_{\rho}(y) }  \left|\Theta\left( \mathbf{a},B_{\rho}(y) \right)(x) \right| \, dx   \leq \delta,
$$
where $$ \Theta\left( \mathbf{a},B_{\rho}(y) \right)(x):= \sup_{\xi \in \R^n \setminus \{ 0\} } \frac{  \left|\mathbf{a}(x,\xi)-(\mathbf{a}(\cdot,\xi))_{B_{\rho}(y)}\right|}{|\xi|^{p-1} }$$
and $$ (\mathbf{a}(\cdot,\xi))_{B_{\rho}(y)} := \mint_{B_{\rho}(y)} \mathbf{a}(x,\xi) \;dx.$$
\end{definition}

The above definition implies that the map $x\mapsto \ba(x,\xi)/|\xi|^{p-1}$ is a locally 
BMO function with the BMO semi-norm  which is
 less than or equals to $\delta$ for all $\xi\in\R^n$.  Then we note that the nonlinearity $\ba$ can be discontinuous for the $x$-variable. In particular, in the model case in \eqref{model}, Definition~\ref{smallbmo} implies that $A(\cdot)$ is a locally BMO function.

\begin{definition}
 Given $\delta \in (0,\frac18)$ and $R>0,$ we say that $\Omega$ is a $(\delta, R)$-Reifenberg flat domain if for every $x \in \partial\Omega$ and every $\rho \in (0, R],$ there exists a coordinate system $\{ y_1, y_2, \dots, y_n\}$ which may depend on $\rho$ and $x,$ such that in this coordinate system $x=0$ and that
$$ B_{\rho}(0) \cap \{ y_n > \delta \rho \} \subset B_{\rho}(0)  \cap \Omega \subset B_{\rho}(0) \cap \{ y_n > -\delta \rho \}.
$$
\end{definition}
In the above definition, $\delta$ is usually supposed to be less than $\frac18$. This number follows from the Sobolev embedding, see for instance \cite{To1}. In this paper, however, it is not important since we will deal with sufficiently small $\delta$ . We remark that the Lipschitz domains with the Lipschitz constant which is less than or equal to $\delta$ belong to the class of $(\delta,R)$-Reifenberg flat domains for some $R>0$. In addition,  the $(\delta,R)$-Reifenberg flat domain $\Omega$ has the following measure density condition:
\begin{equation}\label{dencon}
\sup_{0<\rho\leq R} \sup_{y \in \overline{\Omega}} \frac{\left|B_{\rho}(y)\right|}{\left|\Omega \cap B_{\rho}(y)\right|} \leq \left( \frac{2}{1-\delta} \right)^{n} \leq \left( \frac{16}{7} \right)^n.
\end{equation}
We refer to \cite{BW1,PS1,Re1,To1} for more details on the Reifenberg flat domains and their applications.


%
%
%

Now we present the main results in this paper.  The first result is local $L^q$ estimates in both interior and boundary regions.

\begin{theorem}
\label{mainthm}
Let $1<p<\infty$, $1<s<p^*$, $\gamma_0$ be from \eqref{gamma0}, $\ba:\R^n\times \R^n\to\R^n$ satisfy \eqref{aas1} and \eqref{aas2}, $V:\R^n\to [0,\infty)$ with $V\in L^{\gamma_0}(\Omega)$, and $F\in L^p(\Omega,\R^n)$.    Suppose that the constants $p$, $\gamma$, $q$ satisfy specific conditions given below and that $V\in\mathcal B_\gamma$. 
There exists  a small  $\delta = \delta(n, p, L, \nu,\gamma)>0$ such that if $\mathbf{a}$ is $( \delta, R)$-vanishing, $\Omega$ is a $(\delta, R)$-Reifenberg flat domain for some $R\in(0,1),$ and $u\in  W^{1,p}_0(\Omega)$ is a weak solution to \eqref{maineq}, then for any $x_0\in \overline{\Omega}$ and $r \in (0, \frac{R}{2}]$ we have  the following estimates:
\begin{itemize}
\item[(1)]
If $1<p<\infty$, $\tilde \gamma < \gamma <\infty$ where
\begin{equation}\label{gamma}
   \tilde\gamma :=\max\left\{1,\frac{np}{np-n+p}\right\}=
\left\{ \begin{array}{cl}
 \frac{np}{np-n+p} & \text{if }\ 1<p<n,\\
 1& \text{if }\ p\ge n,
 \end{array}\right. 
\end{equation}
and $1<q<\frac{\gamma^*(p-1)}{p}$,
\begin{equation}\label{mainest1}
 \mint_{\Omega_{ r}(x_0)} |Du|^{pq}\,  dx 
\leq c \left( \mint_{\Omega_{2r}(x_0)} |Du|^p \,dx\right)^q + c  \mint_{\Omega_{2 r}(x_0)}   |F|^{pq} \, dx.
\end{equation}
\item[(2)] If $p\ge 2$, $1<\gamma < \infty$, and $1< q < \gamma$,
\begin{equation}\label{mainest2}
 \mint_{\Omega_{ r}(x_0)}  [V|u|^s]^{q}\,  dx \leq c  \left(\mint_{\Omega_{2r}(x_0)} V|u|^{s} \,dx\right)^q + c  \mint_{\Omega_{2r}(x_0)}   |F|^{pq} \, dx.
\end{equation}
\item[(3)]  If $1< p< 2$, $\frac{n}{p} \le \gamma <\infty$ and $1< q <\gamma$,
\begin{equation}\label{mainest3}
 \mint_{\Omega_{ r}(x_0)}  [V|u|^s]^{q}\,  dx \leq c  \left(\mint_{\Omega_{2r}(x_0)} |Du|^p+V|u|^{s} \,dx\right)^q + c  \mint_{\Omega_{2r}(x_0)}   |F|^{pq} \, dx.
\end{equation}
\end{itemize}
Here the constants  $c>0$ depend on $n,p,L,\nu,s,\gamma,b_{\gamma},q.$ 
\end{theorem}

Regarding the ranges of $\gamma$ and $q$, we will discuss in Remarks~\ref{rmk:range1} and \ref{rmk:range2} below.

\begin{remark} 
In the above theorem, we obtain local $L^q$ estimates for $|Du|^p$ without 
the lower order term  $V|u|^{s}$, and for  $V|u|^{s}$ without the term $|Du|^p$ when $p>2$.
However, when $1<p<2$, the local $L^q$ estimates for $V|u|^{s}$  involve $|Du|^p$, which follows from the approximation lemma, Lemma~\ref{lem:approximation}. As a consequence, the lower bound of $\gamma$ in (3) has to be chosen as not $1$ but $\frac{n}{p}$. This choice insures $\frac{\gamma^*(p-1)}{p}\ge \gamma$, hence we have $|Dh_i|^p + V|h_i|^s \in L^{\gamma}(\Omega_{5\rho_i}(y_i))$ in \eqref{tPhibdd} with Case 3 in the proof of Theorem~\ref{mainthm}.
\end{remark}

%
%
%
%
%
%
%
%

The second result is global $L^q$ estimates.
As a consequence of Theorem~\ref{mainthm}, 
we obtain the following global estimates  by using the standard covering argument in the proof of  \cite[Corollary 2.6]{LO2}, hence we omit the proof. 

\begin{theorem}\label{maincor}
Let $1<p<\infty$, $1<s<p^*$, $\gamma_0$ be from \eqref{gamma0}, $\ba:\R^n\times \R^n\to\R^n$ satisfy \eqref{aas1} and \eqref{aas2}, $V:\R^n\to [0,\infty)$ with $V\in L^{\gamma_0}(\Omega)$, and $F\in L^p(\Omega,\R^n)$.   
Suppose that the constants $p$, $\gamma$, $q$ satisfy specific conditions given below and that $V\in\mathcal B_\gamma$.
 There exists  a small  $\delta = \delta(n, p, L, \nu,\gamma) \in (0,\frac18)$ such that  if $\mathbf{a}$ is $( \delta, R)$-vanishing, $\Omega$ is a $(\delta, R)$-Reifenberg flat domain for some $R\in(0,1),$ and $u\in  W^{1,p}_0(\Omega)$ is a weak solution to \eqref{maineq}, then we have  the following estimates:
\begin{itemize}
\item[(1)]
If $1<p<\infty$, $\tilde\gamma<\gamma<\infty$
where $\tilde\gamma$ is given in \eqref{gamma},
and $1<q < \frac{\gamma^*(p-1)}{p}$, then we have
\begin{equation*}
\Vert Du \Vert_{L^{pq}(\Omega)} \leq c\left(\frac{\mathrm{diam}(\Omega)}{R}\right)^{n(q-1)}\Vert F \Vert_{L^{pq}(\Omega)}.
\end{equation*}

\item[(2)] If $p\ge 2$, $1<\gamma<\infty$ and $1<q<\gamma$, or if $1< p< 2$, $\frac{n}{p}\le \gamma<\infty$ and $1<q<\gamma$,  then we have
\begin{equation*}
\Vert V^{\frac{1}{p}}|u|^{\frac{s}{p}} \Vert_{L^{pq}(\Omega)} \leq c\left(\frac{\mathrm{diam}(\Omega)}{R}\right)^{n(q-1)}\Vert F \Vert_{L^{pq}(\Omega)}.
\end{equation*}
\end{itemize}
Here the constants  $c>0$ depend on $n,p,L,\nu,s,\gamma,b_{\gamma},q.$ 
\end{theorem}

\begin{remark}\label{rmk:range1} 
 If $V\in \mathcal B_{\gamma}$, then $V$ belongs to the $\mathcal B_{\gamma+\epsilon}$ class for some small $\epsilon>0$ from the self improving property of the $\mathcal B_\gamma$ class (see \cite{Ge}). 
Therefore, by considering $\gamma+\epsilon$ instead of $\gamma$ in Theorems \ref{mainthm} and \ref{maincor},  the ranges of $q$ can be extended to $ q\in [1,\frac{\gamma^*(p-1)}{p}]$ and $ q \in [1, \gamma]$, respectively. Note that the case $q=1$ is trivial. On the other hand, as mentioned in Introduction, $L^q$ estimates with $q<1$ is an open problem even in the case $V\equiv 0$.
\end{remark}

\begin{remark}\label{rmk:range2} 
We further comment on the ranges of $\gamma$ and $q$ for the $L^q$ estimates for $|Du|^p$ in Theorem~\ref{mainthm} (1) and Theorem~\ref{maincor} (1). 
\begin{itemize}
\item[(i)] Suppose $p<n$. Then the range of $\gamma$ is $(\frac{np}{np-n+p},\infty)$. On the other hand, for the linear case with $p=2$, we see from  \cite{AB1} that   it is $(1,\infty)$. Therefore, the case $\gamma\in (1,\frac{np}{np-n+p})$ seems missing in our results. However, if $\gamma<\frac{np}{np-n+p}$, we have $\frac{\gamma^*(p-1)}{p}<1$ and the $L^q$ estimate with $q<1$  is the open problem mentioned above. Therefore, our range is best for now.

\item[(ii)] We note that  $\frac{\gamma^*(p-1)}{p}<\gamma$ when $\gamma<\frac{n}{p}$.  For the linear case with $p=2$, in \cite{AB1},
the range of $q$ is $(\frac{1}{2},\gamma)$ when $\gamma<\frac{n}{2}$. Therefore, the range of $q$ in our results could be extended to $1< q<\gamma$ when $\gamma<\frac{n}{p}$.
\end{itemize}
\end{remark}

\begin{remark}
From Remark~\ref{rmk:deltachoice}, we see that the estimates in Theorem~\ref{mainthm} (2)  and Theorem~\ref{maincor} (2) when $p\ge 2$ still hold without the $(\delta,R)$-vanishing condition of  $\ba$ and the $(\delta,R)$-Reifenberg flat condition of  $\Omega$. Moreover,
in Theorem~\ref{mainthm} (1)  and Theorem~\ref{maincor} (1), the dependence $\gamma$ of $\delta$ is replaced with $q$, when $\gamma\ge n$.
\end{remark}

\section{Estimates for homogenous equations}\label{sechomo}

In this section, we prove various regularity estimates for weak solutions to localized equations of \eqref{maineq} with $F\equiv 0$.

We start by recalling interior and boundary Calder\'on-Zygmund estimates for $p$-Laplace type elliptic equations. In particular, we consider non-divergence data. 
 For the following results, we refer to, for instance, \cite{LO1}.

%
%


\begin{lemma}\label{thmDwbdd}
Let $1<p<\infty$ and $\tilde\gamma<\gamma<n$ where $\tilde\gamma$ is given in \eqref{gamma}. There exists a small $ \delta = \delta(n, p, L, \nu,  \gamma) \in (0,\frac18) $ so that if $\mathbf{a}$ is $( \delta, R)$-vanishing and $\Omega$ is a $(\delta,R)$-Reifenberg flat domain for some $R\in(0,1)$, then for any $x_0\in\overline\Omega$, $r\in(0,\frac{R}{2}]$ and for any weak solution $h\in W^{1,p}(\Omega_{2r}(x_0))$ to 
\begin{equation*}
\left\{\begin{array}{rclcc}
-\mathrm{div}\, \mathbf{a}(x,Dh)&  =  &f & \textrm{ in } & \Omega_{2r}(x_0),  \\
h & = & 0 & \textrm{ on } &  \partial_{\mathrm{w}}\Omega_{2r}(x_0)\ \text{if}\ B_{2r}(x_0)\not\subset\Omega.
\end{array}\right.
\end{equation*}
with $f\in L^{\gamma}(\Omega_{2r}(x_0))$, we have
\begin{equation*}
\begin{split}
 \left(\mint_{\Omega_r(x_0)} |Dh|^{\gamma^*(p-1)} \, dx\right)^{\frac{p}{\gamma^*(p-1)}}   
 &\leq  c \mint_{\Omega_{2r}(x_0)} |Dh|^{p}\, dx\\
&\qquad + c \left(   \mint_{\Omega_{2r}(x_0)} |r f  |^{\gamma} \, dx      \right)^{\frac{p}{\gamma(p-1)}}
\end{split}
\end{equation*} 
for some $c=c(n,p,L,\nu, \gamma)>0.$
\end{lemma}

%
%

We need the standard iteration lemma whose proof can be found in, for instance, \cite{HL1}. 
\begin{lemma}\label{teclem}
Let $ g :[a,b] \to \R $ be a bounded nonnegative function. Suppose that for any $\tau_1,\tau_2$ with  $ 0< a \leq \tau_1 < \tau_2 \leq b $,
$$
g(\tau_1) \leq \tau g(\tau_2) + \frac{C_1}{(\tau_2-\tau_1)^{\beta}}+C_2
$$
where $C_1,C_2 \geq 0, \beta >0$ and $0\leq \tau <1$. Then we have
$$
g(\tau_1) \leq c\left(  \frac{C_1}{(\tau_2-\tau_1)^{\beta}}+ C_2 \right)
$$
for some constant $c=c(\beta, \tau) >0$.
\end{lemma}

Now , we consider the following homogeneous
Dirichlet problems of the following type
\begin{equation}\label{eqh}
\left\{\begin{array}{rclll}
-\mathrm{div}\, \mathbf{a}(x,Dh) + V |h|^{s-2}h&  =  &0 & \textrm{in } \ \Omega_{r},&  \\
h& = & 0 & \textrm{on } \   \partial_{\mathrm{w}}\Omega_{r}, \ \text{ if }\ B_{r}\not\subset\Omega,\end{array}\right.
\end{equation}
where $1<p<\infty$, $1<s<p^*$, $\mathbf{a}:\R^n\times \R^n\to\R^n$ satisfies
\begin{equation}\label{aAss}
|\ba(x,\xi)|\leq L|\xi|^{p-1} \ \ \text{and}\ \ \ba(x,\xi)\cdot \xi\geq \nu |\xi|^p, \quad x,\xi \in\R^n,
\end{equation}
for some $0<\nu\leq L$, and $V:\R^n\to [0,\infty)$ does $V\in L^{\gamma_0}(\Omega)$. Note that the assumptions \eqref{aas1} and \eqref{aas2} imply \eqref{aAss}.

We shall need the following Caccioppoli type estimates. For simplicity, we use both the plus-minus sign $\pm$ and the minus-plus sign $\mp$, and they are all linked, that is, we take all upper signs or all lower signs. 

\begin{lemma}\label{lem:caccio} (Caccioppoli estimates)
Under the setting above, let $h\in W^{1,p}(\Omega_{r})$ be a weak solution  to \eqref{eqh}.
Then for every $k\ge 0$, every $B_\rho(y)\subset B_r$ with $B_\rho(y)\cap \Omega_r\neq \emptyset$, and every $\nu\in(0,1)$, we have
\begin{equation}\label{caccio1}
\int_{\Omega_{\nu \rho}(y) }  | D(h \mp k)_{\pm} |^p \,dx + \int_{\Omega_{\nu \rho}(y) } V|h|^{s-2}h_\pm  (h \mp k)_{\pm} \,dx \le c \int_{\Omega_{\rho}(y)} \left[\frac{(h \mp k)_{\pm}}{(1-\nu) r}\right]^p\,dx 
\end{equation}
for some c$\tilde\gamma<\gamma<n,$ where $\tilde\gamma$ is given in \eqref{gamma}. onstant $c = c(n, p, L, \nu) >0.$
\end{lemma}

\begin{proof}
We take $\pm(h \mp k)_{\pm}\eta^p$  as a test function in the weak form of \eqref{eqh}, where $\eta\in C^\infty_0(B_\rho(y))$ such that $0\le \eta\le 1$, $\eta\equiv 1$ on $B_{\nu \rho}(y)$, and $|D\eta|\le \frac{c(n)}{(1-\nu)r}$. Note that $\pm(h \mp k)_{\pm}\eta^p=0$ on $\partial_{\mathrm w}\Omega_r$ since $h=0$ on $\partial_{\mathrm w}\Omega_r$ and $k\ge 0$. Then we have
\[
\begin{split}
\int_{\Omega_{\rho}(y) }  \ba(x,Dh) \cdot  D[\pm (h \mp k)_{\pm}\eta^p ]  \,dx + \int_{\Omega_{ \rho}(y) } V|h|^{s-2}h [\pm   (h \mp k)_{\pm} \eta^{p}] \,dx =0. 
\end{split}
\]
Note that since $k\ge0$,  $Dh=\pm D(h \mp k)_{\pm}$ and $\pm h=h_{\pm}$ when $(h \mp k)_{\pm} >0$.  Moreover, applying  \eqref{aAss}, we find
\[\begin{split}
 \ba(x,Dh) \cdot  D[\pm (h \mp k)_{\pm} \eta^p ]  &= \pm\ba(x,\pm D(h \mp k)_{\pm}) \cdot  \left[\eta^p D(h \mp k)_{\pm} + (h \mp k)_{\pm} D\eta^p  \right]\\
 &\ge   \nu \eta^p|D(h \mp k)_{\pm}|^p - c \eta^{p-1} |D(h \mp k)_{\pm}|^{p-1} \frac{(h \mp k)_{\pm}}{(1-\nu)r}.
\end{split}\] 
Therefore, by the above results and the second inequality in \eqref{aAss}, we have
\[
\begin{split}
\nu \int_{\Omega_{\rho}(y) } \eta^p|D(h \mp k)_{\pm}|^p  \,dx + &\int_{\Omega_{\rho}(y) } V|h|^{s-2}h_\pm   (h \mp k)_{\pm} \eta^{p}\,dx\\
& \le c \int_{\Omega_{\rho}(y)} \eta^{p-1} |D(h \mp k)_{\pm}|^{p-1} \frac{(h \mp k)_{\pm}}{(1-\nu)r}\,dx.
\end{split}
\]
Using Young's inequality and the properties of $\eta$, we have the estimate \eqref{caccio1}.
\end{proof}

Next, we prove the local boundedness of the weak solutions to \eqref{eqh} with $L^\infty-L^q_w$ type estimates.
\begin{lemma}\label{lem:suph}
Under the assumptions of Lemma~\ref{lem:caccio}, 
we have that $h\in L^\infty_{\mathrm{loc}}(B_r)$, where we extend $h$ to $B_r$ by $0$. Moreover, for every $0< r_1 < r_2 \le r$, every $q>0$ and every weight $w\in A_t$ with $t\ge1$,
\begin{equation}\label{Linftyestimate}
\begin{split}
\|h\|_{L^\infty(\Omega_{r_1})} & \leq c \left(\frac{r_2}{r_2-r_1}\right)^{\frac{nt}{q}} \left(\frac{1}{w(B_{r_2})} \int_{\Omega_{r_2}} |h|^q w \, dx \right)^{\frac{1}{q}}\\
&=c \left(\frac{r_2}{r_2-r_1}\right)^{\frac{nt}{q}} \left(\frac{1}{(w)_{B_{r_2}}} \frac{1}{|B_{r_2}|}\int_{\Omega_{r_2}} |h|^q   w \, dx \right)^{\frac{1}{q}}
\end{split}
\end{equation}
for some constant $c = c(n, p, L, \nu,q,t, [w]_{t}) >0$, where $w(B_{r_2}) := \int_{B_{r_2}} w \,dx.$
\end{lemma}

\begin{proof}
Since $V\ge 0$, 
we have from  \eqref{caccio1} that 
 for every $B_\rho(y)\subset B_r$, $k\ge 0$, and $\nu\in(0,1)$, 
\begin{equation*}
\begin{split}
\int_{\Omega_{\nu \rho}(y) }  | D(h \mp k)_{\pm} |^p \,dx  \le c \int_{\Omega_{\rho}(y)} \left[\frac{(h \mp k)_{\pm}}{(1-\nu) r}\right]^p\,dx.
\end{split}\end{equation*}
This implies that  $h$ belongs to the De Giorgi class in \cite[Chapter 7]{Gi1}. Therefore, in view of \cite[Chapter 7.2]{Gi1}, we deduce that 
\begin{equation}\label{Linftyestimate1}
\|h\|_{L^\infty(\Omega_{r_1})} \leq c \left(\frac{r_2}{r_2-r_1}\right)^{\frac{n}{q}}\left( \frac{1}{|B_{r_2}|}\int_{\Omega_{r_2}} |h|^q \, dx \right)^{\frac{1}{q}}
\end{equation}
for every $q>0$, which is the desired estimate \eqref{Linftyestimate} with the trivial weight $w\equiv 1$.
Since $w\in A_t$, we note from \cite[Proposition 9.1.5 (8)]{G1} that for every $B_\rho\subset \R^n$ and every $f\in L^{t}_w(B_{\rho})$,
\[
\mint_{B_{\rho}} |f| \,dx \le [w]_t \left(\frac{1}{w(B_\rho)} \int_{B_\rho} |f|^t w \, dx \right)^{\frac{1}{t}}.
\]  
Therefore, plugging this inequality with $f=|h|^{q/t}$ and $B_\rho=B_{r_2}$ into  \eqref{Linftyestimate1} replacing $q$ by $q/t$, we obtain \eqref{Linftyestimate}.
\end{proof}

Using the above results, we deduce the following lemma, which allows us to change the integrangd $V|h|^s$ to $|Dh|^p$ later.

\begin{lemma}\label{lem:VhDh}
Under the assumptions of Lemma~\ref{lem:caccio},
let  $h\in W^{1,p}(\Omega_{r})$ be a weak solution to \eqref{eqh}. If $V\in B_\gamma$ for some $\gamma >1$, then
\begin{equation}\label{VhDhestimate}
r^\frac{p}{p-1}(V)_{B_{r/2}}^{\frac{p}{s(p-1)}} \left(\frac{1}{|B_{r/2}|}\int_{\Omega_{r/2}}V|h|^s\,dx\right)^{\frac{p(s-1)}{s(p-1)}} \le \frac{c}{|B_r|} \int_{\Omega_{r}}|Dh|^p\,dx
\end{equation}
for some $c=c(n,p,L,\nu,s,\gamma,b_\gamma)>0$.
\end{lemma}
\begin{proof}
We first note from \eqref{equivalence} that $V\in A_t$ with $t\ge 1$ and $[V]_t$ depending on $\gamma$ and $b_\gamma$.
We extend $h$ to $B_{r}$ by $0$. Let $r/2\le r_1 < r_2 \le r$ be arbitrary, $r_3:=\frac{r_1+r_2}{2}$, and $\eta\in C^\infty_0(B_{r_3})$ with $\eta\ge0$, $\eta\equiv 1$ in $B_{r_1}$ and $|D\eta|\le c/(r_2-r_1)$. We take $h\eta^p$ as a test function in the weak form of \eqref{eqh}. Then, in the same argument as in the proof of Lemma~\ref{lem:caccio} we have
\[
  \nu \int_{B_{r_3}}   |Dh |^p\eta^p \,dx + \int_{B_{r_3}}  V|h|^{s}\eta^p \,dx \le  \frac{c}{r_2-r_1} \int_{B_{r_3}} \eta^{p-1}|Dh|^{p-1}|h| \,dx.
\]
Applying H\"older's inequality, it follows that
\[
\int_{B_{r_1 }}  V|h|^{s} \,dx  \le  \frac{c}{r_2-r_1}\left(\int_{B_{r_3}} |Dh|^p \,dx\right)^{\frac{p-1}{p}}  \left(\int_{B_{r_3}} |h|^p \,dx\right)^{\frac{1}{p}} .
\]
We multiply the both sides with 
\[
M:= \left[r^{p} (V)_{B_{r/2}}^{\frac{p}{s}} \left(\mint_{B_{r/2}} V|h|^s\,dx\right)^{\frac{s-p}{s}}\right]^{\frac{1}{p-1}}
\]
to obtain
\[\begin{split}
M\int_{B_{r_1 }}  V|h|^{s} \,dx & \le   \frac{c\,rM^{\frac{1}{p}}}{r_2-r_1}  \left(\int_{B_{r_3}} |Dh|^p \,dx\right)^{\frac{p-1}{p}} \\
&\qquad \times \Bigg[\underbrace{(V)_{B_{r/2}}^{\frac{p}{s}}\left(\mint_{B_{r/2}} V|h|^s\,dx\right)^{\frac{s-p}{s}}\int_{B_{r_3}} |h|^p \,dx}_{=:I}\Bigg]^{\frac{1}{p}} .
\end{split}\]
(Here, we assume that $(V|h|^s)_{B_{r/2}}>0$. If the average is zero, the estimate \eqref{VhDhestimate} is trivial.)
We now estimate $I$. By \eqref{Linftyestimate} with $(q,r_1,w)$ in place of $(s,r_3,V)$, 
\[\begin{split}
I &\le |B_r|(V)_{B_{r/2}}^{\frac{p}{s}}\left(\mint_{B_{r/2}} V|h|^s\,dx\right)^{\frac{s-p}{s}}\|h\|^p_{L^\infty(B_{r_3})}\\
&\le c |B_r|\left(\frac{r}{r_2-r_1}\right)^{\frac{npt}{s}}(V)_{B_{r/2}}^{\frac{p}{s}}\left(\mint_{B_{r/2}} V|h|^s\,dx\right)^{\frac{s-p}{s}}\left(\frac{1}{(V)_{B_{r_2}}}\mint_{B_{r_2}} V|h|^{s}\, dx\right)^{\frac{p}{s}}\\
&\le c \left(\frac{r}{r_2-r_1}\right)^{\frac{npt}{s}}\int_{B_{r_2}} V|h|^{s}\,dx.
\end{split}\]
Inserting this into the preceding estimate and using Young's inequality, we have
\[\begin{split}
M\int_{B_{r_1 }}  V|h|^{s} \,dx & \le  c \left(\frac{r}{r_2-r_1}\right)^{1+\frac{nt}{s}}\left(\int_{B_{r}} |Dh|^p \,dx\right)^{\frac{p-1}{p}}  \Bigg[M\int_{B_{r_2}}  V|h|^{s} \,dx\Bigg]^{\frac{1}{p}}\\
&  \le \frac{1}{2} M\int_{B_{r_2}}  V|h|^{s} \,dx +c \left(\frac{r}{r_2-r_1}\right)^{(1+\frac{nt}{s})\frac{p}{p-1}}\int_{B_{r}} |Dh|^p \,dx.
\end{split}\]
Theorefore, by applying Lemma~\ref{teclem} and recalling the definition of $M$, we get the conclusion.
\end{proof}

Finally, we derive the following two reverse H\"older type higher integrability results for homogeneous equations.

\begin{theorem}\label{lem:estimatehomo0}
Let $1<p<\infty$, $1<s<p^*$, $\gamma_0$ be from \eqref{gamma0},  $\ba:\R^n\times \R^n\to\R^n$ satisfy \eqref{aAss}, and $V:\R^n\to [0,\infty)$ do $V\in L^{\gamma_0}(\Omega)$ and $V\in \mathcal{B}_{\gamma}$ for some  $\gamma>1$. 
If $h\in  W^{1,p}(\Omega_{2r}(x_0))$ is a weak solution to
\[
\left\{\begin{array}{rclcl}
-\mathrm{div}\, \mathbf{a}(x,Dh) + V |h|^{s-2}h&  =  &0 & \textrm{ in } & \Omega_{2r}(x_0),  \\
h & = & 0 & \textrm{ on } &  \partial_{\mathrm{w}}\Omega_{2r}(x_0)\ \text{ if }\ B_{2r}(x_0)\not\subset\Omega,\end{array}\right.
\]
where $x_0\in\overline{\Omega}$, then 
we have 
\begin{equation*}
\left( \mint_{{\Omega}_r(x_0)} \left[ V|h|^s \right]^{\gamma }  \, dx \right)^{\frac{1}{\gamma }}  \leq  c  \mint_{\Omega_{2r}(x_0)} V |  h |^s  \,dx
\end{equation*}
for some $c=c(n,p,\nu,L,s,\gamma,b_\gamma)>0$.
\end{theorem}
\begin{proof}
Note that, by \eqref{equivalence}, $V\in A_t$ with $t\ge 1$ and $[V]_t$ depending on $\gamma$ and $b_\gamma$. Then, the desired estimate directly follows from $V\in \mathcal B_\gamma$ and Lemma~\ref{lem:suph} with $q=s$ and $w=V$ as
\[
\left(\frac{1}{|B_r|}\int_{\Omega_{r}}[V|h|^{s}]^\gamma\,dx\right)^{\frac{1}{\gamma}} 
 \le  b_\gamma (V)_{B_{r}}\|h\|_{L^\infty(\Omega_{r})}^{s} \le  \frac{c}{|B_{2r}|}\int_{\Omega_{2r}} V |h|^s\,dx. \qedhere
\]

\end{proof}

\begin{theorem}\label{lem:estimatehomo}
Let $1<p<\infty$, $1<s<p^*$, $\gamma_0$ be from \eqref{gamma0},  $\ba:\R^n\times \R^n\to\R^n$ satisfy \eqref{aas1} and \eqref{aas2}, and $V:\R^n\to [0,\infty)$ do $V\in L^{\gamma_0}(\Omega)$ and   $V\in \mathcal{B}_{\gamma}$ for some   $\gamma$ 
satisfying  $\tilde\gamma<\gamma<n$, where $\tilde\gamma$ is given in \eqref{gamma}. 
 There exists  a small  $\delta = \delta(n, p, L, \nu,\gamma)>0$ such that the following holds: if $\mathbf{a}$ is $( \delta, R)$-vanishing, $\Omega$ is a $(\delta, R)$-Reifenberg flat domain for some $R>0,$ and $h\in  W^{1,p}(\Omega_{8r}(x_0))$ is a weak solution to
\[
\left\{\begin{array}{rclcl}
-\mathrm{div}\, \mathbf{a}(x,Dh) + V |h|^{s-2}h&  =  &0 & \textrm{ in } & \Omega_{8r}(x_0),  \\
h & = & 0 & \textrm{ on } &  \partial_{\mathrm{w}}\Omega_{8r}(x_0)\ \text{ if }\ B_{8r}(x_0)\not\subset\Omega,\end{array}\right.
\]
where $x_0\in\overline{\Omega}$ and $8r\le R$, then
we have 
\begin{equation*}
\left( \mint_{{\Omega}_r(x_0)} |Dh|^{\gamma^* (p-1)}  \, dx \right)^{\frac{p}{\gamma^* (p-1)}} \leq c  \mint_{\Omega_{8r}(x_0)} |Dh|^p\,dx
\end{equation*}
for some $c=c(n,p,\nu,L,s,\gamma,b_\gamma)>0$.

\end{theorem}

\begin{proof}
For simplicity, we write $\Omega_\rho=\Omega_\rho(x_0)$ with $\rho>0$. We extend $h$ by $0$ to $B_{16r}\setminus \Omega$ if it is nonempty. From the facts that $h\in L^\infty(B_{2r})$ and  $V\in L^\gamma(\Omega_{2r}),$ we see that $V|h|^{s-2}h\in L^\gamma(\Omega_{2r})$. Therefore, applying Lemma~\ref{thmDwbdd} with $f=V|h|^{s-2}h$, we have
\begin{equation}\label{DwDwrVwes}
\begin{split}
 &\left(\mint_{\Omega_r} |Dh|^{\gamma^*(p-1)} \, dx\right)^{\frac{p}{\gamma^*(p-1)}} \\  
 &\hspace{2cm} \leq c  \mint_{\Omega_{2r}} |Dh|^{p}\, dx+ c \left(   \mint_{\Omega_{2r}} \left[r V|h|^{s-1}  \right]^\gamma \, dx\right)^{\frac{p}{\gamma(p-1)}}.
 \end{split}
\end{equation}
We now estimate the second integral on the right hand side of \eqref{DwDwrVwes}. Since $V\in \mathcal B_\gamma$ and $|\Omega_\rho|\approx |B_\rho|$ with $0<\rho<R$, see \eqref{dencon},
\[
\left(\mint_{\Omega_{2r}}[rV|h|^{s-1}]^\gamma\,dx\right)^{\frac{p}{\gamma(p-1)}} 
\le b_\gamma r^{\frac{p}{p-1}} (V)_{B_{2r}}^{\frac{p}{p-1}}\|h\|_{L^\infty(B_{2r})}^{\frac{p(s-1)}{p-1}}.
\]
Moreover, applying Lemma~\ref{lem:suph}, replacing $(r,r_1,r_2,w)$ by $(4r,2r,4r,V)$, and Lemma~\ref{lem:VhDh}, replacing $r$ with $8r$ 
\[\begin{split}
\left(\mint_{\Omega_{2r}}[rV|h|^{s-1}]^\gamma\,dx\right)^{\frac{p}{\gamma(p-1)}} 
&\le b_{\gamma}  r^{\frac{p}{p-1}} (V)_{B_{2r}}^{\frac{p}{p-1}}\|h\|_{L^\infty(B_{2r})}^{\frac{p(s-1)}{p-1}} \\
&\le c r^{\frac{p}{p-1}} (V)_{B_{4r}}^{\frac{p}{p-1}}\left(\frac{1}{(V)_{B_{4r}}}\mint_{B_{4r}}V|h|^s\,dx \right)^{\frac{p(s-1)}{s(p-1)}}  \\
&\le c r^{\frac{p}{p-1}} (V)_{B_{4r}}^{\frac{p}{s(p-1)}}\left(\mint_{B_{4r}}V|h|^s\,dx \right)^{\frac{p(s-1)}{s(p-1)}}  \\
&\le c \mint_{\Omega_{8r}} |Dh|^{p}\, dx.
\end{split}\]
Plugging the preceding estimate  into \eqref{DwDwrVwes}  we have the  desired estimate.
\end{proof}

\begin{remark}
In the above theorems, it is possible that $\gamma<\gamma_0$. In this case, since $V\in L^{\gamma_0}$, we have $V|h|^s\in L^{\gamma_0}(\Omega_r(x_0))$ in Theorem ~\ref{lem:estimatehomo0} and $Dh\in L^{\gamma_0^*(p-1)}(\Omega_r(x_0),\R^n)$ in Theorem~\ref{lem:estimatehomo}. However, the reverse H\"older type estimates can be obtained with the exponent  $\gamma$, by the assumption that $V\in \mathcal{B}_\gamma$. 
\end{remark}

\section{$L^q$ estimates}
\label{sec gradient estimates}
We are now ready to prove  our main results.
\subsection{Comparison} We start with the following comparison lemma. 

\begin{lemma}\label{lem:approximation}
Let $1<p<\infty$, $1<s<p^*$, $\gamma_0$ be from \eqref{gamma0}, $\ba:\R^n\times \R^n\to\R^n$ satisfy \eqref{aas1} and \eqref{aas2}, $V:\R^n\to [0,\infty)$ do $V\in L^{\gamma_0}(\Omega)$, and $F\in L^p(\Omega,\R^n)$.
 If $u \in  W^{1,p}_0(\Omega)$ is the weak solution to \eqref{maineq}, and  
 $h \in W^{1,p}(\Omega_{8r})$ is the weak solution to
\begin{equation}
\label{lhomoeq}
\left\{\begin{array}{rclcc}
-\mathrm{div}\, \mathbf{a}(x,Dh) + V |h|^{s-2}h&  =  & 0 & \textrm{ in } & \Omega_{8r},  \\
h & = & u & \textrm{ on } & \partial \Omega_{8r},
\end{array}\right.
\end{equation}
where $\Omega_{8r}=\Omega_{8r}(x_0)$ with $x_0=\overline{\Omega}$, then we have the following estimates:
\begin{itemize}
\item[(i)] (Energy estimates)
\begin{equation}\label{lDumDvibdd0}
\int_{\Omega_{8r}}  |Dh|^{p}  \, dx \le c \int_{\Omega_{8r}}  |Du|^{p}  \, dx +c \int_{\Omega_{8r}}  |F|^{p}  \, dx.
\end{equation} 

\item[(ii)] If $p\ge2$,
\begin{equation}\label{lDumDvibdd1}
\int_{\Omega_{8r}}  |Du-Dh|^{p}  \, dx \le c \int_{\Omega_{8r}}  |F|^{p}  \, dx,
\end{equation} 
and for every $\epsilon\in(0,1)$,
\begin{equation}\label{lDumDvibdd11}
\int_{\Omega_{8r}}   V |u-h|^{s} \, dx \le  \epsilon \int_{\Omega_{8r}}   V |u|^{s} \, dx+ c(\epsilon)\int_{\Omega_{8r}}  |F|^{p}  \, dx.
\end{equation} 

\item[(iii)]If $1<p<2$,  for every $\epsilon \in (0,1),$ 
\begin{equation}\label{lDumDvibdd2}
\int_{\Omega_{8r}}  |Du-Dh|^{p}  \, dx
\leq \epsilon \mint_{\Omega_{8r}}  |Du|^{p}  \, dx +c(\epsilon) \int_{\Omega_{8r}}  |F|^{p}  \, dx,
\end{equation}
\begin{equation}\label{lDumDvibdd3}
\int_{\Omega_{8r}}  V |u-h|^{s} \, dx
\leq \epsilon \int_{\Omega_{8r}}  |Du|^{p} + V |u|^{s} \, dx + c(\epsilon) \int_{\Omega_{8r}}  |F|^{p}  \, dx.
\end{equation} 
\end{itemize}
Here, $c>0$ depends on $n,p,L,\nu$ and $s$, and $c(\epsilon)>0$ does on $n,p,L,\nu,s$ and $\epsilon$.
\end{lemma}

\begin{proof}
%
We test the equations \eqref{maineq} and \eqref{lhomoeq} with the test function $\varphi:=u-h$ in order to discover
\begin{equation}\label{weakformu-v}\begin{split}
&\int_{\Omega_{8r}} \left( \mathbf{a}(x, Du) - \mathbf{a}(x, Dh) \right) \cdot ( Du-Dh)\, dx \\
&+ \int_{\Omega_{8r}} V \left( |u|^{s-2}u - |h|^{s-2}h \right) \cdot (u-h)\, dx= \int_{\Omega_{8r}} |F|^{p-2}F \cdot (Du-Dh)\, dx.
\end{split}\end{equation}
We recall the monotonicity conditions \eqref{mono} and 
\begin{equation}\label{mono1}
\left( |\xi_2|^{s-2}\xi_2 - |\xi_1|^{s-2}\xi_1 \right) \cdot (\xi_2-\xi_1) \geq c(s)  (|\xi_1|^2+|\xi_2|^2)^\frac{s-2}{2} |\xi_2-\xi_1|^2, \quad 
\xi_1,\xi_2\in\R.
\end{equation}
Then we see that the two terms on the left hand side of \eqref{weakformu-v} are nonnegative.

We first prove (i). By \eqref{weakformu-v}, \eqref{mono} and \eqref{aas1}, we infer
\[
\begin{split}
& \int_{\Omega_{8r}}|Dh|^p\,dx \le c \int_{\Omega_{8r}} \ba(x,Dh)\cdot Dh\,dx \\
 &\le c  \int_{\Omega_{8r}} |\mathbf{a}(x, Du)| (|Du|+|Dh|) + |\mathbf{a}(x, Dh)|  |Du| + |F|^{p-1}(|Du|+|Dh|)\, dx\\
  &\le c  \int_{\Omega_{8r}} |Du|^p + |Du|^{p-1}|Dh| + |Dh|^{p-1}  |Du| + |F|^{p-1}(|Du|+|Dh|)\, dx.
\end{split}\]
Therefore, applying Young's inequality, we obtain \eqref{lDumDvibdd0}.
 
We next prove (ii).  Note that in this case the exponent $s$ satisfies either $s < 2$ or $s \ge 2$. Using the monotonicity conditions \eqref{mono} and \eqref{mono1} together with the fact that $2\le p$,  we derive from \eqref{weakformu-v} that 
\[
\int_{\Omega_{8r}} | Du-Dh |^p\, dx + \int_{\Omega_{8r}} V (|u|^2+|h|^2)^{\frac{s-2}{2}}|u-h|^2\, dx \le c  \int_{\Omega_{8r}} |F|^{p-1}|Du-Dh|\, dx,
\] 
which together Young's inequality implies 
\[
\int_{\Omega_{8r}} | Du-Dh |^p\, dx + \int_{\Omega_{8r}} V (|u|^2+|h|^2)^{\frac{s-2}{2}}|u-h|^2\, dx \le c  \int_{\Omega_{8r}} |F|^{p}\, dx.
\] 
Therefore, we obtain \eqref{lDumDvibdd1}. 
On the other hand, by Young's inequality again, we have that for any $\epsilon\in(0,1)$,
\[
\int_{\Omega_{8r}} V |u-h|^s\, dx \le \epsilon \int_{\Omega_{8r}} V |u|^s+V|h|^s\, dx+c(\epsilon)  \int_{\Omega_{8r}} |F|^{p}\, dx,
\]
from which, by choosing $\epsilon$ so small we first have 
\[
\int_{\Omega_{8r}} V |h|^s\, dx \le c \int_{\Omega_{8r}} V |u|^s\, dx+c \int_{\Omega_{8r}} |F|^{p}\, dx.
\]
Inserting this into the previous estimate yields \eqref{lDumDvibdd11}.

Lastly we prove (iii). Applying the monotonicity conditions \eqref{mono} and \eqref{mono1}  and Young's inequality to \eqref{weakformu-v} yields that for any $\kappa_1\in(0,1)$
\[\begin{split}
&\int_{\Omega_{8r}}\left( |Du|^{2}+ |Dh|^2 \right)^{\frac{p-2}{2}}  | Du-Dh|^2\, dx +\int_{\Omega_{8r}}V \left( |u|^{2}+ |h|^2 \right)^{\frac{s-2}{2}}  | u-h|^2\, dx \\
& \leq c \int_{\Omega_{8r}} |F|^{p-1}|Du-Dh|\, dx  \leq \kappa_1 \int_{\Omega_{8r}} |Du-Dh|^p\, dx +c(\kappa_1) \int_{\Omega_{8r}} |F|^{p}\, dx.
\end{split}\]
Since $1<p<2$, then, by Young's inequality again, we have that for any $\kappa_2\in(0,1),$
\[\begin{split}
 |\xi_2-\xi_1|^p &=   |\xi_2-\xi_1|^p \left( |\xi_2|^2 + |\xi_1|^2 \right)^{\frac{p(p-2)}{4}} \left( |\xi_2|^2 + |\xi_1|^2 \right)^{\frac{p(2-p)}{4}}\\
 &\leq \kappa_2 \left( |\xi_2|^2 + |\xi_1|^2 \right)^{\frac{p}{2}} +c(\kappa_2) \left( |\xi_2|^2 + |\xi_1|^2 \right)^{\frac{p-2}{2}}  |\xi_2-\xi_1|^2,
 \end{split}\]
for all $\xi_1,\xi_2\in \R^d$ with $d\in\mathbb{N}.$ Therefore, using the above results we have
\[\begin{split}
\int_{\Omega_{8r}} | Du-Dh|^p\, dx & \leq c \kappa_2 \int_{\Omega_{8r}} |Du|^p+|Dh|^p\,dx\\
&\quad + \kappa_1 c(\kappa_2)\int_{\Omega_{8r}} |Du-Dh|^p\, dx +c(\kappa_1)c(\kappa_2) \int_{\Omega_{8r}} |F|^{p}\, dx,
\end{split}\]
which implies \eqref{lDumDvibdd2} by choosing $\kappa_1$ and $\kappa_2$ such that $c\kappa_2=\frac{\epsilon}{2}$ and $\kappa_1c(\kappa_2)=\frac{1}{2}$ and applying  \eqref{lDumDvibdd0}. In addition, we also have  
\[\begin{split}
\int_{\Omega_{8r}}& | Du-Dh|^p\, dx +\int_{\Omega_{8r}}V | u-h|^s\, dx \\
&\leq c \kappa_2 \int_{\Omega_{8r}} |Du|^p+|Dh|^p+ V|u|^s+V|h|^s\,dx\\
& \qquad + \kappa_1 c(\kappa_2)\int_{\Omega_{8r}} |Du-Dh|^p\, dx +c(\kappa_1)c(\kappa_2) \int_{\Omega_{8r}} |F|^{p}\, dx.
\end{split}\]
Then we first choose $\kappa_1$ and $\kappa_2$ so small to get   
\begin{equation}\label{V|h|^sbdd}
\int_{\Omega_{8r}} V|h|^s\,dx \le c \int_{\Omega_{8r}}  V|u|^s+|Du|^p+|F|^p\,dx,
\end{equation}
where we used \eqref{lDumDvibdd0}. Inserting this into the previous estimate, we have 
\[\begin{split}
\int_{\Omega_{8r}} &| Du-Dh|^p\, dx +\int_{\Omega_{8r}}V | u-h|^s\, dx \\
&\le c\kappa_2  \int_{\Omega_{8r}} |Du|^p+ V|u|^s + |F|^p\,dx\\
& \quad + \kappa_1 c(\kappa_2)\int_{\Omega_{8r}} |Du-Dh|^p\, dx +c(\kappa_1)c(\kappa_2) \int_{\Omega_{8r}} |F|^{p}\, dx\\
\end{split}\] 
Therefore, again choosing $\kappa_1$ and $\kappa_2$ so small that 
 $c\kappa_2=\epsilon$ and $\kappa_1 c(\kappa_2)=\frac{1}{2} $,
we obtain \eqref{lDumDvibdd3}.
\end{proof}

\begin{remark} 
In the above lemma, by using \eqref{lDumDvibdd11} and \eqref{V|h|^sbdd}, we also obtain
\begin{equation*}
\int_{\Omega_{8r}}  V |h|^{s} \, dx
\leq c \int_{\Omega_{8r}}  |Du|^{p}\chi_{\{1<p<2\}} + V |u|^{s} +  |F|^{p}  \, dx.
\end{equation*}

\end{remark}

\subsection{Proof of Theorem \ref{mainthm}} We prove the theorem by using we the approach introduced by Mingione in \cite{AM1,Min1}. The proof goes in four steps.

\vspace{0.2cm}
\noindent\textit{Step 1. Setting.}

Assume that $\ba:\R^n\times \R^n\to\R^n$ is $(\delta,R)$-vanishing and $\Omega$ is $(\delta,R)$-Reifenberg flat for some $R>0,$ where   $\delta\in(0,1)$  will be chosen sufficiently small later in Step 3 (see Remark~\ref{rmk:deltachoice}). 
Fix any  $x_0\in\overline{\Omega}$ and $r>0$ satisfying $r \leq \frac{R}{2}$. Assume that $V\in \mathcal B_\gamma$, where the range of $\gamma>1$ is given in (1)--(3) of Theorem~\ref{mainthm}.

We prove the estimates \eqref{mainest1}-\eqref{mainest3} at one time by denoting the function $\Phi(v;x)$ and constant $\gamma_1$ differently as follows:
 

\medskip

\noindent 
{\bf Case 1. Estimation of \eqref{mainest1}:} Let $\gamma\in (\tilde\gamma,\infty)$ with $\tilde\gamma$ given in \eqref{gamma}. We fix any 
$q\in(1,\frac{\gamma^*(p-1)}{p})$
 with $\gamma^*$ in \eqref{gamma*}, and denote by  
\[
\Phi(v;x):=|Dv(x)|^{p}
\quad \text{and} \quad 
\gamma_1:=
\left\{ \begin{array}{cl}
 \frac{\gamma^*(p-1)}{p} & \text{if }\ \gamma<n,\\
 \max\{q+1,\left(\frac{3}{2}\right)^*\frac{p-1}{p}\}&  \text{if }\ \gamma \ge n.
 \end{array}\right.
\]
In Case 1, we note that, when $\gamma \ge n$, one can find the constant $\gamma_2\in [\frac{3}{2},n)$ such that $\gamma_1=\frac{\gamma_2^*(p-1)}{p}$. Moreover,  it is clear that $V\in \mathcal B_{\gamma_2}$ since $\gamma_2<n\le \gamma$.

\noindent 
{\bf Case 2. Estimation of \eqref{mainest2}:}  Let $\gamma\in (1,\infty)$. We fix any $q\in(1,\gamma)$, and denote by  
\[
\Phi(v;x):=V(x)|v(x)|^s 
 \quad \text{and} \quad 
 \gamma_1:=\gamma.
\]

\noindent 
{\bf Case 3. Estimation of \eqref{mainest3}}: Let $\gamma\in[\frac{n}{p},\infty)$. We fix any $q\in(1,\gamma)$, and denote by  
\[
\Phi(v;x):=|Dv(x)|^{p}+V|v(x)|^s
\quad \text{and} \quad 
\gamma_1:=\gamma.
\]

\noindent We note that $p\ge 2$  in Case 2 and $1<p< 2$ in Case 3.

\medskip

%

With $\Phi(v;x)$ denoted in  above and the weak solution $u$  to \eqref{maineq}, we define
$$
E(\lambda, \rho) := \{x \in \Omega_{\rho} :   \Phi(u;x)>\lambda  \},\quad \lambda>0,
$$
and 
 \begin{equation}\label{lambda0}
\lambda_0:=  \mint_{\Omega_{2r}} \Phi(u;x) \,dx+\frac{1}{\delta_1} \mint_{\Omega_{2r}} |F|^p \,dx,
\end{equation}
where   $\delta_1\in(0,1)$  will be chosen sufficiently small later in Step 4.

Finally, fix any $\tau_1, \tau_2$ with $1\leq \tau_1 < \tau_2 \leq 2.$ 
Note that $ \Omega_{r}\subset \Omega_{\tau_1 r } \subset  \Omega_{\tau_2 r } \subset  \Omega_{2 r } $.

\vspace{0.2cm}
\noindent\textit{Step 2. Covering argument.}

%
%
%

We consider $\lambda>0$ large enough so that
\begin{equation}\label{lambdarg}
\lambda > \alpha\, \lambda_0,\ \ \textrm{where } \alpha : = \left( \frac{16}{7}\right)^n\left( \frac{80}{\tau_2 - \tau_1}\right)^n.
\end{equation}
Note that
$ \Omega_{\rho}(y) \subset \Omega_{2r}$ for any $y \in E(\lambda, \tau_1 r)$ and any $ \rho \in \left( 0, (\tau_2-\tau_1)\,r \right].$
Then from the measure density condition \eqref{dencon} and  the definition of $\lambda_0$ given in \eqref{lambda0}, we infer that
\[
\begin{split}
\mint_{\Omega_{\rho}(y)}  \Phi(u;x)
  \,dx+
\frac{1}{\delta_1}\mint_{\Omega_{\rho}(y)} |F|^p \,dx
& \le \frac{|\Omega_{2r}|}{|\Omega_{\rho}(y)|} \lambda_0  \le \left( \frac{16}{7}\right)^{n} \left( \frac{2r}{\rho}\right)^{n}\, \lambda_0\\
& \le \alpha\, \lambda_0  < \lambda,
\end{split}
\]
provided that
$$ \frac{(\tau_2 - \tau_1)\,r}{40} \leq \rho \leq (\tau_2-\tau_1)\,r.$$
On the other hand, it follows from Lebesgue's differentiation theorem that for almost every $y \in E(\lambda, \tau_1 r),$
$$ \lim_{\rho \rightarrow 0} \bigg(\mint_{\Omega_{\rho}(y)}\Phi(u;x) \,dx +\frac{1}{\delta_1} 
 \mint_{\Omega_{\rho}(y)} |F|^p \,dx\bigg) > \lambda.$$

Therefore the above inequalities and 
the continuity of the integral with respect to the measure of the domain yield that for almost every $y \in E(\lambda, \tau_1 r),$ there exists
$$ \rho_{y} = \rho(y) \in \bigg( 0, \frac{(\tau_2-\tau_1)\,r}{40}\bigg)$$
such that
$$ \mint_{\Omega_{\rho_y}(y)} \Phi(u;x)   \,dx +\frac{1}{\delta_1}
 \mint_{\Omega_{\rho_y}(y)} |F|^p \,dx = \lambda,$$
and for any $ \rho \in ( \rho_y, (\tau_2-\tau_1)r]$ there holds
$$ \mint_{\Omega_{\rho}(y)} \Phi(u;x)   \,dx +\frac{1}{\delta_1}  \mint_{\Omega_{\rho}(y)} |F|^p \,dx <\lambda.
 $$

 As a consequence, Vitali's covering theorem implies the following:
\begin{lemma}\label{coveringlem}
Given $\lambda > \alpha\, \lambda_0,$ there exists a disjoint family of $\{ \Omega_{\rho_i}(y^i)\}_{i=1}^{\infty}$ with $y^i \in E(\lambda, \tau_1 r)$ and $\rho_{i} \in \left(0, \frac{(\tau_2-\tau_1)\,r}{40} \right)$ such that
$$E(\lambda, \tau_1 r) \subset \bigcup_{i=1}^{\infty} \Omega_{5\rho_i}(y^i), $$
\begin{equation}\label{covering12}
\mint_{\Omega_{\rho_i}(y^i)} \Phi(u;x) \,dx +
\frac{1}{\delta_1} \mint_{\Omega_{\rho_i}(y^i)} |F|^p \,dx = \lambda,
\end{equation}
and  for any $ \rho \in ( \rho_i, (\tau_2-\tau_1)\,r]$,
\begin{equation}\label{covering2}
\mint_{\Omega_{\rho}(y^i)}\Phi(u;x)  \,dx+ \frac{1}{\delta_1} \mint_{\Omega_{\rho}(y^i)} |F|^p \,dx<\lambda.
\end{equation}
\end{lemma}

Furthermore, according to Lemma~\ref{coveringlem} we infer
\begin{equation*}
\begin{split}
 & \left|\Omega_{\rho_i}(y^i)\right| =  \frac{1}{\lambda} \bigg( \int_{\Omega_{\rho_i}(y^i)} \Phi(u;x)  \,dx +
\frac{1}{\delta_1} \int_{\Omega_{\rho_i}(y^i)} |F|^p \,dx \bigg)\\  
 & \leq \frac{1}{\lambda} \bigg( \int_{\Omega_{\rho_i}(y^i)\cap\{\Phi(u;x) >\frac{\lambda}{4} \}} \Phi(u;x) \,dx+ \frac1{\delta_1} \int_{\Omega_{\rho_i}(y^i)\cap\{|F|^p>\frac{\delta_1\lambda}{4} \}} |F|^p \,dx+\frac{\lambda}{2} \left|\Omega_{\rho_i}(y^i)\right|\bigg)\\
 & =\frac{1}{2} \left|\Omega_{\rho_i}(y^i)\right|  + \frac{1}{\lambda}  \bigg( \int_{\Omega_{\rho_i}(y^i)\cap\{ \Phi(u;x)  >\frac{\lambda}{4} \}}  \Phi(u;x)  \,dx+ \frac1{\delta_1} \int_{\Omega_{\rho_i}(y^i)\cap\{|F|^p>\frac{\delta_1\lambda}{4} \}} |F|^p \,dx\bigg)
 \end{split}
\end{equation*}
and so
\begin{equation}\label{omegai}
\begin{split}
 \left|\Omega_{\rho_i}(y^i)\right|\leq  \frac{2}{\lambda} \bigg( \int_{\Omega_{\rho_i}(y^i)\cap\{ \Phi(u;x) >\frac{\lambda}{4} \}} \Phi(u;x)   \,dx+ \frac1{\delta_1} \int_{\Omega_{\rho_i}(y^i)\cap\{|F|^p>\frac{\delta_1\lambda}{4} \}} |F|^p \,dx\bigg).
 \end{split}
\end{equation}

\vspace{0.2cm}
\noindent\textit{Step 3. Comparison estimates.}
%
%

We note from \eqref{covering2} in Lemma~\ref{coveringlem} that
\begin{equation*}
\mint_{\Omega_{40\rho_i}(y^i)} \Phi(u;x)   \,dx+ \frac{1}{\delta_1} \mint_{\Omega_{40\rho_i}(y^i)}|F|^p \,dx<\lambda.
\end{equation*}
Applying Lemma~\ref{lem:approximation}, we have that for any $\epsilon \in (0,1),$  there exists a small $\delta_1= \delta_1( \epsilon, n, p,   L, \nu, s) \in (0,1)$ such that
\begin{equation}\label{tPhimPhibdd}
\begin{split}
\mint_{\Omega_{40\rho_i}(y^i)}  \Phi(u-h_i;x)\,dx 
&\le \epsilon \mint_{\Omega_{40\rho_i}(y^i)} \Phi(u;x) \, dx+c(\epsilon)  \mint_{\Omega_{40\rho_i}(y^i)} |F|^{p}  \, dx\\
& \le \epsilon \lambda + c(\epsilon) \delta_1 \lambda \le 2 \epsilon \lambda.
\end{split}
\end{equation} 
(In fact, when $p\ge 2$ in  {\bf Case 1}, 
$c(\epsilon)$  in \eqref{tPhimPhibdd} does not depend on $\epsilon$.)  
Furthermore, recalling the definition of $\gamma_1$ in {\bf Cases 1--3} and applying  
Theorems~\ref{lem:estimatehomo0} and \ref{lem:estimatehomo} with $\gamma=\gamma_1$, we have 
\begin{equation}\label{tPhibdd}
\begin{split}
\bigg( \mint_{\Omega_{5\rho_i}(y^i)} \Phi(h_i;x)^{\gamma_1}  \, dx \bigg)^{\frac1{\gamma_1}} &\le c \mint_{\Omega_{40\rho_i}(y^i)} \Phi(h_i;x)\,dx\\
& \le c  \mint_{\Omega_{40\rho_i}(y^i)} \Phi(u;x)+  |F|^p \,dx   \le c \lambda,
\end{split}
\end{equation}
where $h_i \in W^{1,p}(\Omega_{20\rho_i}(y^i))$ is the unique weak solution to
\begin{equation*}
\left\{\begin{array}{rclcc}
-\mathrm{div}\, \mathbf{a}(x,Dh_i) + V |h_i|^{s-2}h_i&  =  & 0 & \textrm{ in } & \Omega_{40\rho_i}(y^i),  \\
h_i & = & u & \textrm{ on } & \partial \Omega_{40\rho_i}(y^i).
\end{array}\right.
\end{equation*}

\begin{remark}\label{rmk:deltachoice}
At this stage, the constant $\delta$ is fixed as the one in Theorem~\ref{lem:estimatehomo} with $\gamma=\gamma_1$. Therefore, $\delta$ depends on $n,p,L,\nu,\gamma$ when $\gamma<n$ or $n,p,L,\nu,q$ when $\gamma \ge n$. Moreover, in Case 2, 
Theorem~\ref{lem:estimatehomo} is not used, but only  Theorem~\ref{lem:estimatehomo0} is.
Hence  the $(\delta,R)$-vanishing smallness assumptions on $\ba$ and $\Omega$ are not needed in this case. 
\end{remark}

Let $y \in \Omega_{5\rho_i}(y^i)$ such that $\Phi(u;y)> K\lambda,$ where constant $K\geq 1$ will be chosen later.
We then note that
$$
\Phi(u;y) \le 2^{\max\{p,s\}-1} [\Phi(u-h_i;y)+\Phi(h_i;y)].
$$
Here, we need to consider the two cases:
\[
\textrm{(i)}\, \Phi(h_i;y)  \le \Phi(u-h_i;y), \quad 
\textrm{(ii)}\, \Phi(h_i;y)  > \Phi(u-h_i;y).
\]
For the case (i), it is clear that
$$
\Phi(u;y) \leq 2^{\max\{p,s\}} \Phi(u-h_i;y).
$$
For the case (ii), we see that
$$ K\lambda  <\Phi(u;y)  \le 2^{\max\{p,s\}} \Phi(h_i;y) ,$$
from which, it follows that
$$
\Phi(u;y)  \le 2^{\max\{p,s\}} \Phi(h_i;y) \left[ \frac{2^{\max\{p,s\}}\Phi(h_i;y)}{K\lambda}   \right]^{\gamma_1-1}  = \frac{2^{\gamma_1\max\{p,s\}}}{(K \lambda)^{\gamma_1-1}} \Phi(h_i;y)^{\gamma_1} .
$$
 For the both cases (i) and (ii), we finally obtain that
\[
\begin{split}
\Phi(u;y)&\le   2^{\max\{p,s\}}\Phi(u-h_i;y)+  \frac{2^{\gamma_1\max\{p,s\}}}{(K \lambda)^{\gamma_1-1}}\Phi(h_i;y)^{\gamma_1}
\end{split}
\]
for any $y \in \Omega_{5\rho_i}(y^i)$ such that $\Phi(u;y)> K\lambda.$

Then we apply \eqref{tPhimPhibdd} and \eqref{tPhibdd} to discover
\[
\begin{split}
& \int_{\Omega_{5\rho_i}(y^i) \cap E(K\lambda , \tau_2 r)} \Phi(u;x)\, dx\\
 &\le c \int_{\Omega_{5\rho_i}(y^i)}  \Phi(u-h_i;x) \,dx + \frac{c}{(K \lambda)^{\gamma_1-1}} \int_{\Omega_{5\rho_i}(y^i)} \Phi(h_i;x)^{\gamma_1}  \, dx\\
 & \le  c \left(\epsilon \lambda  +  \frac{\lambda^{\gamma_1}}{(K \lambda)^{\gamma_1-1}}\right) \left| \Omega_{5\rho_i}(y^i)\right| \\
 & \le c \lambda \left(\epsilon   + K^{1-\gamma_1}\right) \left| \Omega_{\rho_i}(y^i)\right|  = c\tilde\epsilon \lambda\left| \Omega_{\rho_i}(y^i)\right|
 \end{split}
\]
for some constant $c=c(n,p,  L, \nu,s,\gamma, b_\gamma)>0,$ where
\begin{equation}\label{tepsilon}
\tilde \epsilon:=   \epsilon + K^{1-\gamma_1}.
\end{equation}
Inserting \eqref{omegai} into the previous estimate, we conclude that
\[
\begin{split}
 &\int_{\Omega_{5\rho_i}(y^i) \cap E(K\lambda , \tau_2 r)}\Phi(u;x)\, dx\\
  & \le c\tilde \epsilon \bigg( \int_{\Omega_{\rho_i}(y^i)\cap\left\{ \Phi(u;x) >\frac{\lambda}{4} \right\}} \Phi(u;x)   \,dx+ \frac1{\delta_1} \int_{\Omega_{\rho_i}(y^i)\cap\left\{|F|^p>\frac{\delta_1\lambda}{4} \right\}} |F|^p \,dx\bigg).
 \end{split}
\]

From Lemma~\ref{coveringlem}, we note that $\Omega_{\rho_i}(y^i)$ is mutually disjoint and
$$
E(K\lambda, \tau_1 r) \subset E(\lambda, \tau_1 r )\subset \bigcup_{i=1}^{\infty} \Omega_{5\rho_i}(y^i) \subset \Omega_{\tau_2 r} ,
$$
since $K \geq 1.$ Then  we obtain that
\begin{equation}\label{EDubddcal}
\begin{split}
 &\int_{ E(K\lambda, \tau_1 r)}\Phi(u;x)\, dx  \le   \sum_{i=1}^{\infty}  \int_{\Omega_{5\rho_i}(y^i) \cap E(K\lambda, \tau_1 r)} \Phi(u;x)\, dx \\
 & \le  c\tilde \epsilon \bigg( \int_{\Omega_{\tau_2r}\cap\left\{ \Phi(u;x)>\frac{\lambda}{4} \right\}} \Phi(u;x) \,dx + \frac1{\delta_1} \int_{\Omega_{\tau_2r} \cap\left\{|F|^p>\frac{\delta_1\lambda}{4} \right\}} |F|^p \,dx\bigg)
 \end{split}
\end{equation}
for some constant  $c=c(n,p, \gamma, L, \nu, b_\gamma)>0.$

\vspace{0.2cm}
\noindent \textit{Step 4. Proof of \eqref{mainest1}--\eqref{mainest3}.} 

We are now ready to conclude the proof via  Fubini's theorem with a truncation argument. For $k >0$,
let us define
$$
\Phi_k(u;x) := \min\left\{\Phi(u;x), k \right\},
$$
and consider the upper level set with respect to $\Phi_k $ as
$$ E_k ( \tilde \lambda, \rho):= \left\{ y \in \Omega_{\rho} :  \Phi_k(u;y)> \tilde \lambda\right\}\ \ \text{for }\tilde \lambda,\rho>0. 
$$
Then since $E_k ( \tilde\lambda, \rho)=\emptyset$ when $k\le \tilde\lambda$ and $E_k ( \tilde\lambda, \rho)= E (\tilde\lambda, \rho)$ when $k> \tilde\lambda,$
it follows from \eqref{EDubddcal} that
\begin{equation*}
\begin{split}
 \int_{ E_k(K\lambda, \tau_1 r)}\Phi(u;x)\, dx   &\le  c\tilde \epsilon \bigg(  \int_{E_k\left(\frac{\lambda}{4},\tau_2r\right)}\Phi(u;x)\,dx + \frac1{\delta_1} \int_{\Omega_{\tau_2r} \cap\left\{|F|^p>\frac{\delta_1\lambda}{4} \right\}} |F|^p \,dx\bigg).
 \end{split}
\end{equation*}
By multiplying both sides by $\lambda^{q-2}$ and integrating with respect to $\lambda$ over $(\alpha\lambda_0, \infty)$, we then have that
\begin{equation}\label{lamPhiEk}
\begin{split}
 &I_0 :=\int_{\alpha\lambda_0}^\infty \lambda^{q-2}\int_{ E_k(K\lambda, \tau_1 r)} \Phi(u;x) \, dxd\lambda\\
 & \leq  c\tilde \epsilon  \bigg(  \int_{\alpha\lambda_0}^\infty \lambda^{q-2} \int_{E_k\left(\frac{\lambda}{4},\tau_2r\right)} \Phi(u;x) \,dxd\lambda +\int_{\alpha\lambda_0}^\infty \lambda^{q-2}  \int_{\Omega_{\tau_2r}\cap\left\{\frac{|F|^p}{\delta_1}>\frac{\lambda}{4} \right\}} \frac{|F|^p}{\delta_1} \,dxd\lambda\bigg)\\
&=: c\tilde\epsilon (I_1+I_2).
 \end{split}
\end{equation}
Here, by virtue of Fubini's theorem, we derive that
\begin{equation*}
\begin{split}
I_0 &= \int_{E_k(K\alpha\lambda_0, \tau_1 r)}
\Phi(u;x)  \bigg( \int_{\alpha\lambda_0}^{\Phi_k(u;x)/K}  \lambda^{q-2} \, d\lambda \bigg) dx\\
&= \frac{1}{q-1}\, \Bigg\{  \int_{E_k(K\alpha\lambda_0, \tau_1 r)}
\Phi(u;x)  \left[\frac{\Phi_k(u;x)}{K}\right]^{q-1}  \,  dx \\
& \qquad\qquad\qquad -(\alpha\lambda_0)^{q-1}\int_{E_k(K\alpha\lambda_0, \tau_1 r)} 
\Phi(u;x) \,dx \Bigg\},
 \end{split}
\end{equation*}
and
\begin{equation*}
\begin{split}
I_1&=     \int_{E_k\left(\frac{\alpha\lambda_0}{4},\tau_2r\right)} \Phi(u;x) \bigg(\int_{\alpha\lambda_0}^{ 4\Phi_k(u;x)}  \lambda^{q-2} \,d\lambda\bigg) dx \\
&\leq \frac{1}{q-1}  \int_{E_k\left(\frac{\alpha\lambda_0}{4},\tau_2r\right)} \Phi(u;x) \left[4\Phi_k(u;x)\right]^{q-1}\, dx \\
&\leq   \frac{4^{q-1}}{q-1}   \int_{\Omega_{\tau_2 r}}  \Phi(u;x)  \Phi_k(u;x)^{q-1} \, dx.
 \end{split}
\end{equation*}
Similarly, we obtain that
\begin{equation*}
\begin{split}
I_2&= \int_{\Omega_{\tau_2r}\cap\left\{\frac{|F|^p}{\delta_1}>\frac{\alpha\lambda_0}{4} \right\}}  \frac{|F|^p}{\delta_1}  \int_{\alpha\lambda_0}^{4|F|^p/\delta_1} \lambda^{q-2}  \,d\lambda dx\\
&\leq \frac{1}{q-1} \int_{\Omega_{\tau_2r}\cap\left\{\frac{|F|^p}{\delta_1}>\frac{\alpha\lambda_0}{4} \right\}}  \frac{|F|^p}{\delta_1}   \left[\frac{4|F|^p}{\delta_1}\right]^{q-1}\, dx \\
&\leq   \frac{4^{q-1}}{q-1}   \int_{\Omega_{\tau_2 r}} \left[  \frac{|F|^{p}}{\delta_1}\right]^q \, dx.
 \end{split}
\end{equation*}
Therefore we insert the previous estimates for $I_0$, $I_1$, $I_2$ into \eqref{lamPhiEk} to discover
\[
\begin{split}
&\int_{E_k(K\alpha\lambda_0, \tau_1 r)} \Phi(u;x)  \Phi_k(u;x)^{q-1}  \,  dx \\
 &  \leq  (K\alpha\lambda_0)^{q-1}\int_{\Omega_{\tau_1 r}} \Phi(u;x) \,dx
\\
&\quad +c\tilde \epsilon K^{q-1} \bigg( \int_{\Omega_{\tau_2 r}}  \Phi(u;x)\Phi_k(u;x)^{q-1} \, dx + \int_{\Omega_{\tau_2 r}}   \left[\frac{|F|^p}{\delta_1} \right]^{q} \, dx \bigg).
\end{split}
\]
We also note that
$$
\int_{\Omega_{\tau_1 r} \setminus E_k(K\alpha\lambda_0, \tau_1 r)} \Phi(u;x)\Phi_k(u;x)^{q-1}\,  dx  \leq  (K\alpha\lambda_0)^{q-1} \int_{\Omega_{\tau_1 r}} \Phi(u;x) \, dx.
$$
In turn, by the last two estimates we obtain
\[
\begin{split}
&\int_{\Omega_{\tau_1 r}}\Phi(u;x)\Phi_k(u;x)^{q-1}  \,  dx \\& \le (K\alpha\lambda_0)^{q-1}\int_{\Omega_{\tau_1 r}}\Phi(u;x)\,dx\\
&\quad + c_2\tilde \epsilon K^{q-1} \bigg( \int_{\Omega_{\tau_2 r}}  \Phi(u;x) \Phi_k(u;x)^{q-1} \, dx + \int_{\Omega_{\tau_2 r}}   \left[\frac{|F|^p}{\delta_1} \right]^{q} \, dx \bigg)
\end{split}
\]
for some $c_2=c_2(n,p,L,\nu,s,\gamma,b_\gamma,q)>0$. At this stage, we recall the definition of $\tilde \epsilon$ given in \eqref{tepsilon}, and then take large $K>1$ and small $\epsilon\in(0,1)$ depending on  $n,p,L,\nu,s,\gamma,b_\gamma,q$ such that
\[
K\geq (4c_2)^{\frac{1}{\gamma_1-q}} \ \ \ \text{and}\ \ \ \epsilon\leq\frac{1}{4c_2K^{q-1}},
\]
hence $\delta_1=\delta_1(n,p,L,\nu,\gamma,b_\gamma,q)\in(0,1)$ is finally determined.  Recalling the definition of $\alpha$ in \eqref{lambdarg} we  consequently obtain
\[
\begin{split}
&\int_{\Omega_{\tau_1 r}}\Phi(u;x) \Phi_k(u;x)^{q-1}  \,  dx\\
&  \le \frac12 \int_{\Omega_{\tau_2 r}}  \Phi(u;x)\Phi_k(u;x)^{q-1} \, dx +\frac{c\lambda_0^{q-1}}{(\tau_2-\tau_1)^{n}}\int_{\Omega_{2r}}\Phi(u;x)\,dx+ c \int_{\Omega_{2 r}}   |F|^{pq} \, dx .
\end{split}
\]
Then we apply Lemma~\ref{teclem} to discover
$$
\int_{\Omega_{r}}\Phi(u;x) \Phi_k(u;x)^{q-1}  \,  dx  \leq c\lambda_0^{q-1}\int_{\Omega_{2r}}\Phi(u;x)\,dx+ c \int_{\Omega_{2 r}}   |F|^{pq} \, dx
$$
for any $k > 0$. Finally, by virtue of Lebesgue's monotone convergence theorem, H\"older's inequality and Young's inequality together with the definition of $\lambda_0$ in  \eqref{lambda0}, we derive that
\begin{eqnarray*}
\mint_{\Omega_{r}}\Phi(u;x)^q\,dx &= &\lim_{k\to\infty} \mint_{\Omega_{r}}\Phi(u;x) \Phi_k(u;x)^{q-1}  \,  dx\\
& \leq& c\lambda_0^{q-1}\mint_{\Omega_{2r}}\Phi(u;x)\,dx+ c \mint_{\Omega_{2 r}}   |F|^{pq} \, dx\\
& \leq& c\bigg(\mint_{\Omega_{2r}}\Phi(u;x)\,dx\bigg)^{q}+ c \mint_{\Omega_{2 r}}   |F|^{pq} \, dx.
\end{eqnarray*}
This implies  the desired estimates \eqref{mainest1}--\eqref{mainest3}, by
 recalling the definition of $\Phi(u;x)$ in {\bf Cases 1--3}.

\bibliographystyle{amsplain}

\end{document}